\documentclass[preprint,12pt]{elsarticle}

\usepackage{natbib}

\usepackage{graphicx}%
\usepackage{multirow}%
\usepackage{amsmath,amssymb,amsfonts}%
\usepackage{amsthm}%
\usepackage{mathrsfs}%
\usepackage[title]{appendix}%
\usepackage{xcolor}%
\usepackage{textcomp}%
\usepackage{manyfoot}%
\usepackage{booktabs}%
\usepackage{algorithm}%
\usepackage{algorithmicx}%
\usepackage{algpseudocode}%
\usepackage{listings}%
\usepackage{caption}
\usepackage[utf8]{inputenc}
\usepackage{comment}

\begin{document}

\newcommand{\Z}{{\mathbb Z}}
\newcommand{\N}{{\mathbb N}}
\newcommand{\Hi}{{\mathbb H}}
\newcommand{\R}{{\mathbb R}}
\newcommand{\Q}{{\mathbb Q}}
\newcommand{\ICG}{\mathrm{ICG}}
\newcommand{\WCG}{\mathrm{WCG}}
\newcommand{\WICG}{\mathrm{WICG}}

\newtheorem{theorem}{\bf Theorem}[section]
\newtheorem{corollary}[theorem]{\bf Corollary}
\newtheorem{lemma}[theorem]{\bf Lemma}
\newtheorem{proposition}[theorem]{\bf Proposition}
\newtheorem{conjecture}[theorem]{\bf Conjecture}
\newtheorem{remark}[theorem]{\bf Remark}
\newtheorem{problem}[theorem]{\bf Problem}
\newtheorem{definition}[theorem]{\bf Definition}

\newcommand{\QED} {\hfill$\square$}

\def\slika #1{\begin{center} \epsffile{#1} \end{center}}

\begin{frontmatter}

\title{Further results on the lower bound on reduced Zagreb index of trees}

\author[address1]{Milan Ba\v si\'c\corref{mycorrespondingauthor}}
\cortext[mycorrespondingauthor]{Corresponding author}
\ead{basic\_milan@yahoo.com}
\address[address1]{University of Ni\v s, Ni\v s, Serbia}

\author[address2]{Aleksandar Ili\'c}
\ead{aleksandari@gmail.com}
\address[address2]{Meta Inc, Menlo Park, USA}

\begin{abstract}
For a graph $G$, the general reduced second Zagreb index is defined as
$$GRM_\lambda (G) = \sum_{uv \in E} (deg(u) + \lambda) (deg(v) + \lambda),$$
where $\lambda$ is an arbitrary real number and $deg (v)$ is the degree of the vertex $v$.

In this paper, we extend and correct the equality results from [N. Dehgardia, S. Klav\v zar, {\it Improved lower bounds on the general reduced second Zagreb index of trees}, preprint (2023)] regarding the minimal value of $GRM_\lambda$ for $\lambda \geq -1$ among trees with $n$ vertices and a maximal degree $\Delta$. Furthermore, we complement these results with two distinct approaches to determine the minimum value of the general reduced second Zagreb index for molecular trees with $\Delta = 3$ and $\Delta = 4$ in $\lambda = -2$, and characterize the extremal trees.
\end{abstract}


\end{frontmatter}


\section{Introduction}

Vertex degree-based (VDB) topological indices have emerged as central tools in both mathematical graph theory and chemical graph theory because to their ability to encode structural information of molecular graphs. These indices are particularly useful for modeling the physicochemical properties of compounds and developing quantitative structure-property and activity relationships (QSPR / QSAR). Among them, the \emph{first} and \emph{second Zagreb indices}, introduced by Gutman and Trinajstić in the 1970s \cite{gutman1972,gutman1975}, are among the earliest and most extensively studied VDB indices. Over the years, numerous other indices such as the Randić index, the Forgotten index, the Sombor index, and many others have been introduced, each capturing various structural aspects of graphs \cite{gutman2013,furtula2018,nikolic2003}.

More recently, Furtula et al.~\cite{furtula2014} proposed the \emph{reduced second Zagreb index}, defined as
\[
RM_2(G) = \sum_{uv \in E(G)} (\deg(u) - 1)(\deg(v) - 1),
\]
which reflects the structural discrepancy between the second and first Zagreb indices.
The reduced second Zagreb index has been further studied in several works, focusing on its extremal properties and applications to various graph classes \cite{buyantogtokh2020,gao2020,shafique2017}.

Motivated by this, Horoldagva et al.~\cite{horoldagva2019} introduced the \emph{general reduced second Zagreb index} (GRM), given by
\[
GRM_\lambda(G) = \sum_{uv \in E(G)} (\deg(u) + \lambda)(\deg(v) + \lambda),
\]
where $\lambda$ is an arbitrary real number and $deg (v)$ is the degree of the vertex $v$.
This invariant generalizes both $M_2(G)$ and $RM_2(G)$, and satisfies the identity
\[
GRM_\lambda(G) = M_2(G) + \lambda M_1(G) + \lambda^2 |E(G)|.
\]
The index $GRM_\lambda$ unifies various Zagreb-type indices and allows for a flexible parametric framework that has found relevance in both theoretical and applied studies.

The study of GRM has progressed rapidly. Horoldagva et al.~\cite{horoldagva2019} initiated the investigation of extremal properties of $GRM_\lambda$ for graphs with a fixed number of cut edges, providing sharp upper bounds and characterizing extremal graphs for $\lambda \geq -\frac{1}{2}$. Later, Buyantogtokh et al.~\cite{das2022} extended this investigation by deriving lower and upper bounds of $GRM_\lambda$ for trees, unicyclic graphs, and graphs with prescribed girth, independence number, or chromatic number. Their results highlighted the role of structural graph features in determining extremal behavior under $GRM_{\lambda}$, including sharp characterizations for Turán graphs, complete multipartite graphs, and graphs with cut vertices. 

In a more recent work, the GRM index was also examined in the context of graph operations. In \cite{khoeilar2021}, the authors computed $GRM_\lambda$ for Cartesian products, corona products, joins, and other graph compositions, further demonstrating the robustness of GRM under structural transformations.


In this paper, we contribute to this line of research by extending and correcting the equality conditions presented in \cite{dehgardia2023} regarding the minimum value of $GRM_\lambda$ for $\lambda \geq -1$ among trees of given order $n$ and maximum degree $\Delta$. Furthermore for $\lambda = -2$, we study the extremal behavior of $GRM_\lambda$ among molecular trees with $\Delta = 3$ and $\Delta = 4$, providing two distinct techniques for identifing the extremal configurations.

\section{Minimum value of $GRM_{\lambda}$ in the class of trees for $\lambda\geq -1$}

For $\Delta = 2$ the unique tree is a path $P_n$, with 
$$
GRM_{\lambda} (P_n) = (2+\lambda)(n \lambda + 2n -\lambda - 4)
$$

For $\Delta = n - 1$ the unique tree is a star $S_n$ 
$$
GRM_{\lambda} (S_n) = (n - 1) (n - 1 + \lambda) (1 + \lambda).
$$

Let $SP(n, \Delta)$ be a spider with $n$ vertices, maximal degree $\Delta \geq 3$ and at most one leg of length more than one. It holds
$$
GRM_{\lambda} (SP(n, \Delta)) = 
(n \lambda + 2 n - \Delta \lambda - \Delta - 3)(2 + \lambda) + (\Delta - 1)(\Delta + \lambda)(1 + \lambda).
$$

Let $BR(n, \Delta, \Delta')$ be a broom with $n$ vertices, obtained from a path $P_{n - \Delta - \Delta' + 2}$ by attaching $\Delta - 1$ pendent vertices to one endpoint of the path and $\Delta' - 1$ pendent vertices to the other endpoint. The maximum degree of $BR(n, \Delta, \Delta')$ is equal to $\Delta \geq \Delta' \geq 2$ and $n \geq \Delta + \Delta'$. Note that $SP(n, \Delta) \equiv BR (n, \Delta, 2)$.


\begin{theorem}
Let $\lambda \geq -1$ and $n \geq 4$ and $3 \leq \Delta \leq n - 2$. For a tree $T$ on $n$ vertices and the maximal vertex degree $\Delta$, it holds
$$
GRM_{\lambda} (T) \geq
(n \lambda + 2 n - \Delta \lambda - \Delta - 3)(2 + \lambda) + (\Delta - 1)(\Delta + \lambda)(1 + \lambda).
$$
The equality holds iff $T \equiv SP(n, \Delta)$ for $\lambda \geq -1$, and $T \equiv BR(n, \Delta, \Delta')$ for $\lambda = -1$ with $\Delta \geq \Delta' \geq 2$.
\end{theorem}

\begin{proof}
The proof is based on the induction of $n$. For the base case $n = \Delta + 2$, the only such tree is a star $S_{n-1}$ with a pendent vertex attached to a vertex of degree 1.

Assume that $T$ is an arbitrary tree on $n$ vertices, with a fixed vertex $x$ of degree $3 \leq \Delta \leq n - 2$. Since $T \not \equiv S_n$, let $v$ be a pendent vertex that is adjacent to a vertex $w \neq x$.

Let $T' = T - v$ be a tree with $n - 1$ vertices obtained by removing the pendent vertex $v$. Obviously, the maximum degree of the vertex of $T'$ remains $\Delta$. By definition of the general reduced second Zagreb index, we get
\begin{eqnarray}
\label{GRM_leaf}
GRM_{\lambda}(T) = GRM_{\lambda}(T') + (\lambda+1)(\lambda+\deg(w)) + \sum_{(w, w') \in E(T')} (\lambda + \deg(w')).
\end{eqnarray}

By expanding the last summation, the following holds
$$
 GRM_{\lambda}(T) - GRM_{\lambda}(T') \geq (\lambda+1)(\lambda+\deg(w)) + (\deg(w)-2)(\lambda+1)+(\lambda+\deg(u))
$$
with equality iff all neighbors of $w$ have degree 1 except for exactly one vertex $u$ that has degree greater than or equal to 2.

After rearranging the right-hand side of the previous inequality, we get
 \begin{eqnarray*}
 GRM_{\lambda}(T) - GRM_{\lambda}(T') &\geq&
 (\lambda+2)^2 + 2 (\deg(w)-2)(\lambda+1)+(\deg(u)-2) \\
 &\geq& (\lambda+2)^2,
\end{eqnarray*}
with equality iff in addition to the above it holds $\lambda=-1$ or $\deg(w)=2$, and the degree of $u$ being equal to 2.

By applying the induction hypothesis,
$$
GRM_{\lambda}(T') \geq GRM_{\lambda}(SP(n-1, \Delta)),
$$
and finally get
$$
GRM_{\lambda}(T) \geq GRM_{\lambda}(SP(n-1, \Delta)) + (\lambda+2)^2 = GRM_{\lambda}(SP(n, \Delta)),
$$
which implies that $T$ is a spider with at most one leg having a length greater than one.

The equality holds if and only if all neighbors of $w$ have a degree equal to 1, except for the vertex $u$ such that $\deg(u) = 2$, and $\deg(w) = 2$ or $\lambda = -1$. By the induction hypothesis, it is necessary to reattach a pendent vertex $v$ to the vertex $w \neq x$ of $T'$.

Assume that $\deg(w) = 2$ and $\lambda > -1$. Since $\deg(u) = 2$, the only possibility of such attachment to $T' \equiv SP (n-1, \Delta)$ is at a pendent node $w$ of the long leg, which means that $T \equiv SP (n, \Delta)$. 

Assume that $\lambda = -1$. From the induction hypothesis, a pendent vertex $v$ can be reattached to the vertex of degree $\Delta'$ in $BR(n - 1, \Delta, \Delta')$ as long as $\Delta > \Delta'$. It finally follows that for the special case $\lambda = -1$ the equality holds for $T \equiv SP(n, \Delta)$ or $T \equiv BR(n, \Delta, \Delta')$ with $n \geq \Delta + \Delta' + 1$ and $\Delta \geq \Delta'$.
\end{proof}

Note that in \cite{dehgardia2023}, the extremal trees for the edge case of $\lambda = -1$ were not completely identified.

\label{S:1}

Let $m_{i,j}$ be the number of edges in a graph $G$ where the incident vertices have degrees $i$ and $j$.
By $(i,j)$ we denote an edge of a graph whose incident vertices have degrees $i$ and $j$.
Furthermore, we use $n_i$ to represent the number of vertices with a degree of $i$.

\bigskip

\section{Minimum value of $GRM_{-2}$ in the class of molecular trees with  maximal degree 3}

Let $T^1_{opt}(k)$ represent a unique tree of order $n = 3k + 1$, consisting of a path $P_{2k+1} = v_1 v_2 \ldots v_{2k+1}$ with exactly one pendent vertex attached to each vertex $v_i$ for even $1 \leq i \leq 2k+1$.

Based on established relationships involving $m_{i,j}$ and $n_i$, it is possible to calculate the values
$n_1=k+2$, $n_2=k-1$, $n_{3}=k$, $m_{13}=k+2$ and $m_{23}=2k-2$, for $T_{\text{opt}}^1(k)$ with the order $n=3k+1$.
Considering that only the edges $(1,3)$ contribute to the sum $GRM_{-2}(T_{\text{opt}}^1(k))$, the resulting value is therefore

\begin{eqnarray}
\label{Zagreb_opt}
GRM_{-2}(T_{\text{opt}}^1(k))=-(k+2).
\end{eqnarray}

Let $T^2_{opt} (k)$ denote a family of $\lfloor k / 2 \rfloor$ trees of order $n = 3k+2$, composed of $T^{1}_{opt} (k)$ by subdividing one edge $v_iv_{i+1}$, for odd $3\leq i \leq 2k-1$.

Based on established relationships involving $m_{i,j}$ and $n_i$, it is possible to calculate the values
$n_1=k+2$, $n_2=k$, $n_{3}=k$, $m_{13}=k+2$, $m_{22}=1$ and $m_{23}=2k-2$, for $T_{\text{opt}}^2(k)$ with the order $n=3k+2$.
Considering that only the edges $(1,3)$ contribute to the sum $GRM_{-2}(T_{\text{opt}}^2(k))$, the resulting value is therefore

\begin{eqnarray}
\label{Zagreb_opt_2}
GRM_{-2}(T_{\text{opt}}^2(k))=-(k+2).
\end{eqnarray}

Let $T^3_{opt}(k)$ denote a family of trees of order $n = 3k + 3$, obtained from $T^2_{opt}(k)$ either by subdividing an edge $v_i v_{i+1}$ such that $deg(v_i) \geq 2$ and $deg(v_{i+1}) = 2$, or by attaching two pendent vertices to $T^1_{opt}(k)$ at $v_1$, or by attaching a single pendent vertex to $T^2_{opt}(k)$ at a vertex  $v_i$ or $v_{i+1}$ for which $deg(v_i) = deg(v_{i+1}) = 2$. 


Based on established relationships involving $m_{i,j}$ and $n_i$, it is possible to compute the values
$n_1=k+2$, $n_2=k+1$, $n_{3}=k$, $m_{13}=k+2$, $m_{22}=2$ and $m_{23}=2k-2$, for $T_{\text{opt}}^3(k)$ of order $n=3k+3$, as defined in the first case.
Considering that only the edges $(1,3)$ contribute to the sum $GRM_{-2}(T_{\text{opt}}^3(k))$, the resulting value is therefore

\begin{eqnarray}
\label{Zagreb_opt_3}
GRM_{-2}(T_{\text{opt}}^3(k))=-(k+2).
\end{eqnarray}

Similarly, the parameters of trees in $T_{\text{opt}}^3(k)$ with order $n=3k+3$ can be determined according to the second and third case. These parameters are given as follows: $n_1=k+3$, $n_2=k-1$, $n_3=k+1$, $m_{13}=k+3$, $m_{33}=1$, and $m_{23}=2k-2$. The corresponding value of $GRM_{-2}$ is given by $-(k+2)$.

\bigskip

The expression for $GRM_{\lambda}(T)$ can be restated as follows

\begin{eqnarray*}
GRM_{\lambda}(T) &=& \sum_{1\leq i\leq j\leq \Delta} m_{i,j} (\lambda+i)(\lambda+j)\\
  &=& \sum_{1\leq i\leq j\leq \Delta}(\lambda^2+\lambda (i+j)+ij) m_{i,j}.
\end{eqnarray*}
We observe that $GRM_{-2}(T)$ can be expressed as 
\begin{eqnarray}
\label{GRM_{-2}(T)}
GRM_{-2}(T)=\sum_{1\leq i\leq j\leq \Delta}a_{ij} m_{i,j},
\end{eqnarray}

where $a_{ij} = 4 - 2(i+j) + ij$.

Given that  $a_{13} = -1$, $a_{33} = 1$ and $a_{2j}=0$, we have that 
$$
GRM_{-2}(T)=m_{33}- m_{13},
$$
for any tree $T$ with a maximal degree of $\Delta=3$. Our task is to determine the minimum value of $m_{33}- m_{13}$.
Indeed, the following assertion is true.

\begin{theorem}
Let $T$ be a tree on $n \geq 7$ vertices and maximum degree 3. The following inequality holds
$$GRM_{-2}(T) =  m_{33} - m_{13} \geq -(k + 2),$$
where equality holds if and only if
\begin{itemize}
\item if $n = 3k+ 1$, then $T \equiv T^1_{opt} (k)$,
\item if $n = 3k+ 2$, then $T \equiv T^2_{opt} (k)$,
\item if $n = 3k+ 3$, then $T \equiv T^3_{opt} (k)$.
\end{itemize}
\end{theorem}

\begin{proof}
The proof is based on the induction of $n$ with four tree transformations $T \rightarrow T'$. 
The base cases $n = 7, 8, 9$  can be proven directly by analyzing all small molecular trees.

Next, we describe a sequence of graph transformations applied in a given order, from 1 to 4, to the tree $T$ of order $n=3k+r$. The objective is to obtain a transformed tree $T'$ of order $m=3l+r_1$, $m<n$, and $r,r_1\in\{1,2,3\}$ ensuring that $GRM_{-2}(T) \geq -(k+2)$, given the assumption that $GRM_{-2}(T') \geq -(l+2)$.

\medskip

{\bf Transformation 1.} If $m_{22} > 0$, we can merge two adjacent vertices of degree $2$ in the tree $T$, resulting in a new tree denoted as $T'$. Since an edge of type $(2,2)$ does not contribute to the overall value of $GRM_{-2} (T)$, and all other edges remain unchanged, it follows that $GRM_{-2} (T) = GRM_{-2} (T')$.

\medskip

{\bf Transformation 2.} 
If $m_{12} > 0$, assume that the pendent vertex $v$ is a neighbor of $u$ with degree 2, which has another neighbor $w$.
If $w$ has degree 2, then $T$ contains edges of type $(2, 2)$, then Transformation 1 can be applied, resulting in a transformed tree $T'$,  implying that $GRM(T) = GRM(T')$.
If the degree of $w$ is equal to 1, then $T$ corresponds to the path $P_3$, which is not possible since $n \geq 7$.
On the other hand, if the degree of $w$ is equal to 3, removal of the vertex $v$ from $T$ results in a tree $T'$ that contains an additional edge of type $(1,3)$ while preserving the number of edges of type $(3,3)$ as in the original tree $T$. Consequently, the relation $GRM_{-2}(T) = GRM_{-2}(T') + 1$ holds.
\medskip

Consider the longest path in $T$, where $u$ is one of its endpoint vertices, namely, a leaf vertex of $T$, and $v$ is its unique adjacent vertex. The vertex $v$ must have degree 3; otherwise, applying Transformation 2 would produce a tree of a smaller order and consequently decrease $GRM_{-2} (T)$.
Furthermore, at least one of the two other neighbors of $v$ must have degree 1; otherwise, this would contradict the assumption that the longest path of $T$ ends at the vertex $u$. Let $u'$ denote the neighbor of $v$ of degree 1, and let $w$ be the neighbor of $v$ with a degree greater than 1.
\medskip

{\bf Transformation 3.} If the vertex $w$ has degree 3, we can remove the vertices $u$ and $u'$. The resulting tree, denoted $T'$, contains two fewer vertices than $T$. The tree $T'$ is obtained from $T$ by removing two edges of type $(1,3)$, specifically the edges $uv$ and $u'v$, and modifying one edge of type $(3,3)$, namely $vw$, into an edge of type $(1,3)$. Consequently, we conclude that $GRM(T) = GRM(T')$.

\medskip

{\bf Transformation 4.} Assume that $w$ has degree 2, and denote the other neighbor of $w$ as $t$. The degree of $t$ must be equal to $3$. If the degree of $t$ were $2$, Transformation 2 would be applied. The case where the degree of $t$ is $1$ is not possible, as it contradicts the assumption that $n \geq 7$. Now we can remove the vertices $v$, $u'$ and $w$ from the tree $T$, and then connect the pendent vertex $u$ to $t$. The new tree $T'$ has three vertices less than $T$, and the newly added edge creates one edge of type $(1, 3)$ while keeping the degree of $t$ 3. Therefore, $GRM_{-2}(T')$ has three vertices less than $T$ and $GRM_{-2}(T)=GRM_{-2}(T')-1$ . \medskip

By applying one of the first three transformations, it is possible to remove either one or two vertices of $T$ while ensuring that $GRM_{-2}(T)$ remains unchanged or decreases. Consequently, the induction hypothesis can be applied to the resulting tree $T'$. Specifically, according to the induction hypothesis, we have $GRM_{-2}(T')\geq -(l+2)$. Since the order of $T'$ is lower than that of $T$, it follows that $l\leq k$. Moreover, since $GRM_{-2}(T)$ is maintained or reduced, it necessarily follows that $GRM_{-2}(T)\geq GRM_{-2}(T')$. Combining these observations, we conclude that
$GRM_{-2}(T)\geq  GRM_{-2}(T')\geq -(l+2)\geq -(k+2)$.

By applying Transformation 4, we find that the order of $T'$ is given by $3(k-1)+r$. Consequently, according to the induction hypothesis, this implies $GRM_{-2}(T')\geq -(k-1+2)= -(k+1)$. Furthermore, using the relation $GRM_{-2}(T)=GRM_{-2}(T')-1$, we deduce that
$GRM_{-2}(T)\geq -(k+1)-1=-(k+2)$.

\medskip

The equality follows easily from Transformations 1, 3 and 4. 

If any of Transformations 1 to 3 is applied, we conclude that the equality holds if and only if $GRM_{-2}(T) = GRM_{-2}(T')$ and $l = k$.
The condition $GRM_{-2}(T) = GRM_{-2}(T')$ is satisfied when either Transformation 1 or Transformation 3 is applied.

If Transformation 1 is applied, then the order of the tree $T'$ is given by $n - 1 = 3k + (r - 1)$, where $1 \leq r \leq 3$. Since $l = k$, it follows that $r - 1 \geq 1$, which implies $1 \leq r - 1 \leq 2$.
For $r-1 = 1$, by the induction hypothesis, we have $T' \equiv T^1_{opt} (k)$, from which it follows that $T \equiv T^2_{opt} (k)$.
For $r-1 = 2$, by the induction hypothesis, we have $T' \equiv T^2_{opt} (k)$, from which it follows that $T$ is belongs to the class of $T^3_{opt} (k)$.

If Transformation 3 is applied, then the order of the tree $T'$ is given by $n - 2 = 3k + (r - 2)$, where $1 \leq r \leq 3$. Given that $l = k$, it follows that $r - 2 \geq 1$, which implies $ r - 2 =1$.
According to the induction hypothesis, we have $T' \equiv T^1_{opt} (k)$, from which it follows that $T$ is belongs to the class of $T^3_{opt} (k)$.

If Transformations 4 is applied, we conclude that the equality holds if and only if $GRM_{-2}(T) = GRM_{-2}(T')-1$ and $l = k-1$.
According to the induction hypothesis, we have $T' \equiv T^i_{opt} (k)$, from which it follows that $T\equiv T^i_{opt} (k+1)$, for $1\leq i\leq 3$.
\end{proof}







The minimum value of $GRM_{-2}(T)$, where $T$ belongs to the class of trees of order $n$ and maximum degree $\Delta = 3$, can be determined using an algebraic approach. This method can serve as a foundation for identifying the minimum value of $GRM_{-2}(T)$ when $T$ belongs to the class of trees of order $n$ and maximum degree $\Delta = 4$.

\begin{theorem}
Let $T$ denote a tree with maximum degree $3$ and order $n\geq 7$. 
It follows that $GRM_{-2}(T)\geq -\lfloor \frac{n-1}{3} \rfloor-2=-(k+2)$. 
\end{theorem}
\begin{proof}
The following equations hold for any tree $T$ with a maximal degree of $3$

\begin{eqnarray}
    &&n_1 + n_2 + n_3  = n \nonumber\\ 
    &&n_1 + 2n_2 + 3n_3 = 2n-2 \nonumber\\ 
    &&m_{12} + m_{13}  = n_1 \label{1st}\\
    &&m_{12} + 2m_{22} + m_{23}  = 2n_2 \nonumber \\ 
    &&m_{13} + m_{23} + 2m_{33}  = 3n_3 \label{second}\\  \nonumber
\end{eqnarray}

Using the previous system, we observe that we can express the variables $n_1$, $n_2$, $m_{12}$, $m_{13}$, and $m_{33}$ in terms of $n_3$, $m_{22}$, and $m_{23}$.

\begin{eqnarray}
    n_1 &=& 2 + n_3\label{2nd}\\
    n_2 &=& n - 2 - 2n_3\label{second'}\\
    m_{33}&=&n-n_3-3-m_{22} - m_{23} \label{third}\\
    m_{13}&=&5n_3+2m_{22} + m_{23}-2n+6 \label{4th}\\
    m_{12}&=&2n-4n_3-2m_{22} - m_{23}-4\label{5th}.\\ \nonumber
\end{eqnarray}

By subtracting   (\ref{third}) from (\ref{4th}) we can deduce that 
\begin{eqnarray}
    m_{13}-m_{33}&=& 6n_3-3n+3m_{22}+2m_{23}+9 \label{6th}.  \\ \nonumber
 \end{eqnarray}

From (\ref{5th}),  it can be concluded that 
$m_{23}\leq  2n-4n_3-2m_{22} -4$, under the assumption that $m_{12} \geq 0$.
Combining this inequality with equation (\ref{6th}) we deduce that
\begin{eqnarray}
m_{13}-m_{33}\leq n-2n_3-m_{22}+1.\label{7th}
\end{eqnarray}
Now let $n=3k+r$, for $1\leq r\leq 3$.

\medskip

Suppose that $n_3\geq \lfloor \frac{n-1}{3} \rfloor+1=k+1$.

Using (\ref{7th}) and the fact that $n_3\geq k+1$ we obtain that 
\begin{eqnarray}
\label{inequality2}
m_{13}-m_{33}&\leq& (3k+r)-2(k+1)-m_{22}+1\leq k+r-1\leq k+2.\\\nonumber
\end{eqnarray}
The equality holds if and only if $m_{12}=m_{22}=0$ and $n_3=k+1$.

\medskip

Suppose that $n_3\leq \lfloor \frac{n-1}{3} \rfloor=k$.

According to (\ref{1st}) and (\ref{2nd}) we get that
\begin{eqnarray}
\label{inequality1}
    m_{13}-m_{33}&\leq & n_1=n_3+2\leq k+2.\\\nonumber
\end{eqnarray}


\end{proof}

Based on the proof of the theorem, it is possible to determine the degree sequences of trees for which the equality $GRM_{-2}(T) = -(k-2)$ holds.

In (\ref{inequality1}), equality holds if and only if $m_{12}=m_{33}=0$, $n_3=k=\lfloor\frac{n-1}{3} \rfloor$, and $n_1=m_{13}=k+2=\lfloor \frac{n-1}{3} \rfloor+2$. Referring to (\ref{second}), we deduce that $m_{23}=3n_3-m_{13}-2m_{33}=2k-2$. Furthermore, employing (\ref{second'}), we derive $n_2=n-2-2n_3=3k+r-2-2k=k+r-2$. Finally, given that (\ref{third}) holds, we conclude that

 $$
 m_{22}=n-n_3-3-m_{23}-m_{33}=3k+r-k-3-2k+2=r-1.
 $$

However, it can be concluded that equality holds in (\ref{inequality2}) when $r=3$. In fact, equality is achieved when $m_{12}=m_{22}=0$ and $n_3=k+1= \lfloor \frac{n-1}{3} \rfloor+1$. With $n_1=n_3+2=k+3$, using (\ref{1st}) yields $m_{13}=n_1-m_{12}=k+3$. According to (\ref{third}) and (\ref{5th}), it is deduced that $m_{23}+m_{33}=n-n_3-3=2k-1$ and $m_{23}=2n-4n_3-4=2k-2$, respectively. Finally, $n_2=n-2-2n_3=n-2k-4$ is confirmed, thus characterizing all degree sequences and types of edges optimal trees possess concerning minimum $GRM_{-2}(T)$.

It can be easily verified by an induction on $k$ that the optimal tree $T$ satisfies one of the following:
$T\equiv T^1_{\text{opt}}(k)$ or $T\equiv T^2_{\text{opt}}(k)$ or $T\equiv T^3_{\text{opt}}(k)$.

\bigskip

\section{Minimum value of $GRM_{-2}$ in the class of molecular trees with  maximal degree 4}

The following equations hold for any tree with maximal degree $4$ of order~$n$ 

\begin{align}\label{sistem}
\begin{split}
    &n_1 + n_2 + n_3 + n_4 = n,\\
    &n_1 + 2n_2 + 3n_3 + 4n_4 = 2(n-1),\\
    &m_{12} + m_{13} + m_{14} = n_1,\\
    &m_{12} + 2m_{22} + m_{23} + m_{24} = 2n_2,\\
    &m_{13} + m_{23} + 2m_{33} + m_{34} = 3n_3,\\
    &m_{14} + m_{24} + m_{34} + 2m_{44} = 4n_4.\\
\end{split}
\end{align}

Using the system above, we note the possibility of expressing the variables $n_1$, $n_2$, $n_4$, $m_{14}$, $m_{24}$, and $m_{33}$ in relation to the remaining variables.

\begin{eqnarray}
    n_1 &=& -\frac{m_{12}}{2} - \frac{m_{13}}{4} - \frac{m_{22}}{2} - \frac{m_{23}}{4} +\frac{m_{34}}{4} + \frac{m_{44}}{2} + \frac{n}{2} + \frac{n_3}{4} + \frac{3}{2}\label{prva}\\
    n_2 &=& \frac{3m_{12}}{4} + \frac{3m_{13}}{8} + \frac{3m_{22}}{4} + \frac{3m_{23}}{8} - \frac{3m_{34}}{8} - \frac{3m_{44}}{4} + \frac{n}{4} - \frac{7n_3}{8} - \frac{5}{4}\label{druga}\\
    n_4 &=& -\frac{m_{12}}{4} - \frac{m_{13}}{8} - \frac{m_{22}}{4} - \frac{m_{23}}{8} + \frac{m_{34}}{8} + \frac{m_{44}}{4} + \frac{n}{4} - \frac{3n_3}{8} - \frac{1}{4}\\
    m_{14}&=&-\frac{3m_{12}}{2} - \frac{5m_{13}}{4} - \frac{m_{22}}{2} - \frac{m_{23}}{4} + \frac{m_{34}}{4} + \frac{m_{44}}{2} + \frac{n}{2} + \frac{n_3}{4} + \frac{3}{2} \\
    m_{24}&=&\frac{m_{12}}{2} + \frac{3m_{13}}{4} - \frac{m_{22}}{2} - \frac{m_{23}}{4} - \frac{3m_{34}}{4} - \frac{3m_{44}}{2} + \frac{n}{2} - \frac{7n_3}{4} - \frac{5}{2}\label{peta} \\
     m_{33}&=&-\frac{m_{13}}{2}- \frac{m_{23}}{2}-\frac{m_{34}}{2}+\frac{3n_3}{2}.\label{sesta}\\ \nonumber
\end{eqnarray}

Using the formula (\ref{GRM_{-2}(T)}) and given that  $a_{13} = -1$, $a_{14}=-2$, $a_{33} = 1$, $a_{34} = 2$, $a_{44} = 4$ and $a_{2j}=0$, we find that 
$$
-GRM_{-2}(T)= m_{13}+2m_{14}-m_{33}-2m_{34}-4m_{44},
$$
for any tree $T$ with a maximal degree of $\Delta=4$. Now, substituting $m_{14}$ and $m_{33}$ from the relations above into the formula of $-GRM_{-2}(T)$ we obtain that
\begin{eqnarray}
\label{main2}
-GRM_{-2}(T)=-3m_{12}- m_{13}-m_{22}-m_{34}-3m_{44}+n-n_3+3.
\end{eqnarray}

We can now conclude that $GRM_{-2}(T)$ attains its minimum when the conditions $m_{12}=m_{13}=m_{22}=m_{34}=m_{44}=n_3=0$ are satisfied, which further implies that $m_{23}=m_{33}=0$. The minimum value obtained under these conditions is $GRM_{-2}(T)=-(n+3)$.
By considering equations (\ref{prva})–(\ref{peta}), this minimum is achieved for the values $n_1=\frac{n+3}{2}$, $n_2=\frac{n-5}{4}$, $n_4=\frac{n-1}{4}$, $m_{14}=\frac{n+3}{2}$, and $m_{24}=\frac{n-5}{2}$. Consequently, these trees exist for orders satisfying the condition $n\equiv 1 \pmod 4$. Expressing $n$ in the form $n=4k+1$ provides a more convenient representation: $n_1=2k+2$, $n_2=k-1$, $n_4=k$, $m_{14}=2k+2$, and $m_{24}=2k-2$.

\smallskip

Let $TT^1_{opt}(k)$ denote the unique tree of order $n = 4k + 1$, which consists of a path $P_{2k+1} = v_1 v_2 \ldots v_{2k+1}$, where exactly two pendant vertices are attached to each vertex $v_i$ for every even index $1 \leq i \leq 2k+1$. By applying induction on $k$, it can be demonstrated that any tree $T$ with the degree sequence given by $n_1=2k+2$, $n_2=k-1$, $n_4=k$, $m_{14}=2k+2$, $m_{24}=2k-2$, and satisfying $n_3=m_{12}=m_{22}=m_{44}=0$, is isomorphic to $TT^1_{opt}(k)$. For $k=1$, it is evident that $T$ corresponds to the star $S_4$, which is, in fact, the optimal tree $TT^1_{opt}(1)$.

Now, assume that any tree $T'$ of order $n = 4k + 1$ with the given degree sequence is isomorphic to $TT^1_{opt}(k)$. Consider an arbitrary tree $T$ of order $n = 4(k+1) + 1$ with the degree sequence $n_1=2(k+1)+2$, $n_2=k$, $n_4=k+1$, $m_{14}=2(k+1)+2$, $m_{24}=2k$, and satisfying $n_3=m_{12}=m_{22}=m_{44}=0$.
Consider the longest path in $T$, where $u$ is one of its endpoint vertices, namely a leaf vertex of $T$, and $v$ is its unique adjacent vertex. The vertex $v$ must have degree 4.
Furthermore,  two of the three remaining neighbors of $v$ must have a degree of 1. If this were not the case, it would contradict the assumption that the longest path in $T$ terminates at the vertex $u$. The fourth neighbor must have degree 2. 
If we remove vertex $v$ along with all its neighboring vertices that are leaves, we obtain a tree $T'$ of order $n = 4k + 1$. Specifically, we remove three edges of the type $(1,4)$ and one edge of the type $(2,4)$, and we replace one edge of the type $(2,4)$ with an edge of type $(1,4)$.  Consequently, in the transition from $T$ to $T'$, the number of edges of both types $(1,4)$ and $(2,4)$ decreases by 2.
 As a result, the number of edges of type $(1,4)$ is given by $m'_{14} = m_{14} - 2 = 2k + 2$, while the number of edges of type $(2,4)$ is given by $m'_{24} = m_{24} - 2 = 2k - 2$. Moreover, we have actually removed three vertices of degree 1, one vertex of degree 4, and modified the degree of one vertex from 2 to 1. Consequently, during the transition from $T$ to $T'$, the number of leaves decreases by 2, the number of vertices of degree 2 decreases by 1, and the number of vertices of degree 4 decreases by 1 as well. This implies that $n'_1 = n_1 - 2 = 2k + 2$, $n'_2 = n_2 - 1 = k - 1$, and $n'_4 = n_4 - 1 = k$.
By the induction hypothesis, we have that $T' \equiv TT^1_{opt}(k)$. After adding the vertex $v$ along the three leaves attached to it in the described manner, we conclude that $T\equiv TT^1_{opt}(k+1)$.

\bigskip

In the preceding discussion, we have demonstrated that $GRM_{-2}(T)$ attains its minimum value of $-(n+3) = -(4k+4)$ if and only if $n = 4k + 1$ and $T \equiv TT^1_{opt}(k)$.
Therefore, according to (\ref{main2}), if $n \neq 4k + 1$, it follows that $GRM_{-2}(T) \geq -(n+2)$.

If $n=4k+2$, the function $GRM_{-2}(T)$ attains the value $-(n+2)=-(4k+4)$ in (\ref{main2}) if and only if all summands with negative signs are equal to zero, except for exactly one.
Since the conditions $m_{13}>0$ or $m_{34}>0$ imply that $n_3>0$, two possible cases arise: either $n_3=1$ while all other summands remain zero, or $m_{22}=1$ while all other summands remain zero.
In the case where $n_3=1$, it follows that $m_{33}=0$. Otherwise, there would be at least two vertices of degree 3, which contradicts the initial assumption. Consequently, we conclude that $m_{23}>0$. Furthermore, by applying equation (\ref{sesta}), we deduce that $m_{23}=3$. Finally, based on the (\ref{druga}), we obtain that $n_2=\frac{n}{4}-1$, which contradicts that fact that $n_2$ is an integer, as $n\equiv 4\pmod 2$.

If $m_{22}=1$, then, by analyzing equations (\ref{prva})–(\ref{peta}), we deduce that the optimal tree $T$ has the following parameters: $n_1=2k+2$, $n_2=k$, $n_4=k$, $m_{22}=1$, $m_{14}=2k+2$, and $m_{24}=2k-2$.
If $TT^2_{opt} (k)$ represents a family of $\lfloor k / 2 \rfloor$ trees of order $n = 4k+2$, obtained from $TT^{1}_{opt} (k)$ by subdividing an edge $v_i v_{i+1}$ for odd values of $i$ in the range $3 \leq i \leq 2k-1$, then, by applying the induction previously established for the case $n \equiv 1 \pmod{4}$, it can be proved that any tree with the aforementioned parameters belongs to $T^2_{opt} (k)$.
Therefore, according to (\ref{main2}), if $n \not\equiv  \{1,2\}\pmod 4$, it follows that $GRM_{-2}(T) \geq -(n+1)$.

\bigskip

If $n=4k+3$, an analysis of the function $GRM_{-2}(T)$ in (\ref{main2}), following a similar case-based approach as in the previous considerations, leads to the conclusion that it attains the value $-(n+1)=-(4k+4)$ in (\ref{main2}) if and only if all summands with negative signs are equal to zero, except for $m_{22}=2$.
By analyzing equations (\ref{prva})–(\ref{peta}), we deduce that the optimal tree $T$ possesses the following parameters: $n_1=2k+2$, $n_2=k+1$, $n_4=k$, $m_{22}=2$, $m_{14}=2k+2$, and $m_{24}=2k-2$.
Let $TT^3_{opt}(k)$ denote a family of trees of order $n = 4k + 3$, obtained from $TT^2_{opt}(k)$  by subdividing an edge $v_i v_{i+1}$ such that $deg(v_i) \geq 2$ and $deg(v_{i+1}) = 2$.
By applying the previously established induction used in the cases when $n \equiv \{1,2\} \pmod{4}$, it can be shown that any tree with the aforementioned parameters belongs to $TT^3_{opt}(k)$.
Therefore, based on equation (\ref{main2}), if $n \not\equiv \{1, 2, 3\} \pmod{4}$, it follows that $GRM_{-2}(T) \geq -n$.

\bigskip

If $n = 4k + 4$, an analysis of the function $GRM_{-2}(T)$ in equation (\ref{main2}), using a case-based approach similar to the previous considerations, leads to the conclusion that it takes the value $-n = -(4k + 4)$ in equation (\ref{main2}) if one of the following conditions is satisfied: either all summands with negative signs are equal to zero, except for $m_{22} = 3$, or all summands with negative signs are equal to zero, except for $m_{44} = 1$.
By analyzing equations (\ref{prva})–(\ref{peta}), we conclude that the optimal tree, denoted as $T$, has the following parameters in the first case: $n_1 = 2k + 2$, $n_2 = k + 2$, $n_4 = k$, $m_{22} = 3$, $m_{14} = 2k + 2$, and $m_{24} = 2k - 2$.
In the second case, the optimal tree $T$ has the following parameters: $n_1 = 2k + 4$, $n_2 = k - 1$, $n_4 = k + 1$, $m_{44} = 1$, $m_{14} = 2k + 4$, and $m_{24} = 2k - 2$.
Let $TT^4_{opt}(k)$ denote a family of trees of order $n = 4k + 4$, obtained from $TT^3_{opt}(k)$ either by subdividing an edge $v_i v_{i+1}$ such that $deg(v_i) \geq 2$ and $deg(v_{i+1}) = 2$, or by attaching three pendent vertices to $TT^1_{opt}(k)$ at $v_1$, or by attaching two pendent vertices to $TT^2_{opt}(k)$  at one of the vertices $v_i$ or $v_{i+1}$ for which $deg(v_i) = deg(v_{i+1}) = 2$.
By applying the previously established induction used in the cases when $n \equiv \{1,2,3\} \pmod{4}$, it can be concluded that a tree with the any of aforementioned parameters belongs to $TT^4_{opt}(k)$.

\bigskip



%
%

\vspace{0.5cm}

\section{Conclusion}

In this paper, we extended and corrected the known lower bounds for the general reduced second Zagreb index $GRM_\lambda(T)$ among trees of given order $n$ and maximum degree $\Delta$, particularly refining the extremal conditions for the case $\lambda \geq -1$. Our results complete the characterization of the extremal trees for $\lambda = -1$, which was previously only partially resolved in \cite{dehgardia2023}. In addition, for $\lambda = -2$ and trees with maximum degree $\Delta = 3$, we established sharp bounds and a complete structural characterization of the extremal trees via two complementary approaches. The case $\Delta = 4$ was also addressed using a similar algebraic technique, yielding precise results within that class.

Both inductive and algebraic methods used for the case $\lambda = -2$ required careful case analysis and intricate structural arguments. Extending these techniques to trees with arbitrary maximum degree $\Delta$ would introduce significantly more complexity, as the number of structural cases grows rapidly. Consequently, determining the minimal value of $GRM_{-2}(T)$ for general $\Delta \geq 5$ remains an open and challenging problem that we leave for future investigation.

\section*{Declaration of competing interest}

The authors declare that they have no known competing financial
interests or personal relationships that could have appeared to
influence the work reported in this paper.

\section*{Data availability}

No data was used for the research described in the article.

\section*{Acknowledgments}

Authors gratefully acknowledge support from the Research Project
of the Ministry of Education, Science and Technological Development of the Republic of Serbia (number 451-03-137/2025-03/ 200124).

\end{document}